\documentclass[11pt]{amsart}

\usepackage{amssymb}
\usepackage{mathrsfs}
\usepackage{amsfonts}
\usepackage{latexsym,amsmath,amsthm,amssymb,amsfonts}
\usepackage[usenames]{color}
\usepackage{xspace,colortbl}
\usepackage{graphicx}
\usepackage{tipa}

\newcommand{\be}{\begin{equation}}
\newcommand{\ee}{\end{equation}}
\newcommand{\beq}{\begin{eqnarray}}
\newcommand{\eeq}{\end{eqnarray}}

\newtheorem{prop}{Proposition}[section]

\newtheorem{remark}[prop]{Remark}

\def\begeq{\begin{equation}}
\def\endeq{\end{equation}}

\def\odot{\setbox0=\hbox{$\bigcirc$}\relax \mathbin {\hbox
to0pt{\raise.5pt\hbox to\wd0{\hfil $\wedge$\hfil}\hss}\box0 }}

\numberwithin{equation} {section}

\numberwithin{equation}{section}
\newtheorem{theorem}{\bf Theorem}[section]

\newtheorem{lemma}[theorem]{\bf Lemma}

\newtheorem{corollary}[theorem]{\bf Corollary}

\allowdisplaybreaks

\begin{document}

\title[an anisotropic inverse mean curvature flow for spacelike graphic curves]
 {An anisotropic inverse mean curvature flow for spacelike graphic curves in Lorentz-Minkowski plane $\mathbb{R}^{2}_{1}$}

\author{
 Ya Gao,~~ Chenyang Liu,~~ Jing Mao$^{\ast}$}

\address{
 Faculty of Mathematics and Statistics, Key Laboratory of
Applied Mathematics of Hubei Province, Hubei University, Wuhan
430062, China. }

\email{1109452431@qq.com, Echo-gaoya@outlook.com, jiner120@163.com}

\thanks{$\ast$ Corresponding author}

\date{}
\begin{abstract}
In this paper, we consider the evolution of spacelike graphic curves
defined over a piece of hyperbola $\mathscr{H}^{1}(1)$, of center at
origin and radius $1$, in the $2$-dimensional Lorentz-Minkowski
plane $\mathbb{R}^{2}_{1}$ along an anisotropic inverse mean
curvature flow with the vanishing Neumann boundary condition, and
prove that this flow exists for all the time. Moreover, we can show
that, after suitable rescaling, the evolving spacelike graphic
curves converge smoothly to a piece of hyperbola of center at origin
and prescribed radius, which actually corresponds to a constant
function defined over the piece of $\mathscr{H}^{1}(1)$, as time
tends to infinity.
\end{abstract}

\maketitle {\it \small{{\bf Keywords}: Anisotropic inverse mean
curvature flow, spacelike curves, Lorentz-Minkowski space, Neumann
boundary condition.}

}

\section{Introduction}

Throughout this paper, let $\mathbb{R}^{2}_{1}$ be the
$2$-dimensional Lorentz-Minkowski space with the following
Lorentzian metric
\begin{eqnarray*}
\langle\cdot ,\cdot\rangle_{L}=dx_{1}^{2}-dx_{2}^{2}.
\end{eqnarray*}
In fact, $\mathbb{R}^{2}_{1}$ is an 2-dimensional Lorentz manifold
with index 1. Denote by
\begin{eqnarray*}
\mathscr{H}^{1}(1)=\{(x_{1},x_{2})\in\mathbb{R}^{2}_{1}\mid
x_{1}^{2}-x_{2}^{2}=-1 \ \mathrm{and}\  x_{2}>0\},
\end{eqnarray*}
which is exactly the hyperbola of center $(0,0)$ (i.e., the origin
of $\mathbb{R}^{2}$) and radius $1$ in $\mathbb{R}^{2}_{1}$.
Clearly, from the Euclidean viewpoint, $\mathscr{H}^{1}(1)$ is one
component of a hyperbola of two arms.

In this paper, we consider the evolution of spacelike curves
(contained in a prescribed convex sector domain) along an
anisotropic inverse mean curvature flow (IMCF for short), and can
prove the following main conclusion.

\begin{theorem}\label{main1.1}
Let $\alpha < 0$, $M\subset\mathscr{H}^{1}(1)$ be some convex curve
segment of the hyperbola $\mathscr{H}^{1}(1)\subset
\mathbb{R}^{2}_{1}$, and
$\Sigma:=\{rx\in\mathbb{R}^{2}_{1}|r>0,x\in\partial M\}$. Let
$X_{0}:M\rightarrow \mathbb{R}^{2}_{1}$ such that $M_{0}:=X_{0}(M)$
is a compact, strictly convex spacelike $C^{2,\gamma}$-curve
($0<\gamma<1$) which can be written as a graph over $M$. Assume that
\begin{eqnarray*}
M_{0}={\rm graph}_{M}u_{0}
\end{eqnarray*}
is a graph over $M$ for a positive map
$u_{0}:M\rightarrow\mathbb{R}$ and
\begin{eqnarray*}
\partial M_{0}\subset\Sigma,\qquad\left<\mu\circ X_{0},\nu_{0}\circ X_{0}\right>_{L}|_{\partial M}=0,
\end{eqnarray*}
where $\nu_{0}$ is the past-directed timelike unit normal vector of
$M_{0}$, $\mu$ is a spacelike vector field defined along
$\Sigma\cap\partial M=\partial M$ satisfying the following property:
\begin{itemize}
\item For any $x\in\partial M$, $\mu(x)\in T_{x}M$, and moreover, $\mu(x)=\mu(rx)$.
\end{itemize}
Then we have:

(i) There exists a family of strictly convex spacelike curves
$M_{t}$ given by the unique embedding
\begin{eqnarray*}
X\in C^{2+\gamma,1+\frac{\gamma}{2}}(M\times
[0,\infty),\mathbb{R}^{2}_{1})\cap C^{\infty}(M\times
(0,\infty),\mathbb{R}^{2}_{1})
\end{eqnarray*}
with $X(\partial M,t)\subset \Sigma$ for $t\geq 0$, satisfying the
following system
\begin{equation}\label{Eq1}
\left\{
\begin{aligned}
&\frac{\partial}{\partial t}X=\frac{1}{|X|^{\alpha}k}\nu \qquad
&&~\mathrm{in}~
M\times(0,\infty)\\
&\left<\mu\circ X,\nu\circ X\right>_{L}=0  \qquad&&~\mathrm{on}~ \partial M\times(0,\infty)\\
&X(\cdot,0)=M_{0} \qquad &&~\mathrm{in}~M
\end{aligned}
\right.
\end{equation}
where $k$ is the curvature of $M_{t}:= X(M,t)=X_{t}(M)$, $\nu$ is
the past-directed timelike unit normal vector of $M_{t}$, and
$|X|:=|\left<X,X\right>_{L}|^{\frac{1}{2}}$ is the norm of $X$
induced by the Lorentzian metric of $\mathbb{R}^{2}_{1}$.


(ii) The leaves $M_{t}$ are spacelike graphs over $M$, i.e.,
\begin{eqnarray*}
M_{t}={\rm graph}_{M}u(x,t).
\end{eqnarray*}

(iii)  Moreover, the evolving spacelike curves converge smoothly
after rescaling to a piece of $\mathscr{H}^{1}(r_{\infty})$, where
$r_{\infty}$ satisfies
\begin{eqnarray} \label{radius}
\frac{1}{\mathop{{\rm sup}}\limits_{M}
u_{0}}\cdot\frac{\mathcal{L}(M_{0})}{\mathcal{L}(M)} \leq r_{\infty}
\leq \frac{1}{\mathop{{\rm inf}}\limits_{M}
u_{0}}\cdot\frac{\mathcal{L}(M_{0})}{\mathcal{L}(M)},
\end{eqnarray}
where
$\mathscr{H}^{1}(r_{\infty}):=\{r_{\infty}x\in\mathbb{R}^{2}_{1}|x\in\mathscr{H}^{1}(1)\}$,
$\mathcal{L}(M)$ and $\mathcal{L}(M_{0})$ stand for the length of
spacelike curves $M$, $M_0$ respectively.
\end{theorem}

\begin{remark} \label{remark1.1}
{\rm (1) This work was firstly announced by us in a previous work
\cite{gm1}, which actually corresponds to the higher dimensional
case of the work here. Besides, we also mentioned this work in
series works \cite{glm-1,gm3} later. As pointed out in \cite[Section
1]{gm3}, our work here can be seen as the anisotropic version of the
lower dimensional case of the IMCF in \cite{GaoY2} and the inverse
Gauss curvature flow (IGCF for short) in \cite{gm3} simultaneously.
Of course, the IMCF in \cite{GaoY2} and the IGCF in \cite{gm3}
should be imposed the zero Neumann boundary condition (NBC for
short). Based on this fact and the experience on the study of the
anisotropic IMCF with zero NBC in the $(n+1)$-dimensional
Lorentz-Minkowski space $\mathbb{R}^{n+1}_{1}$ (both the lower
dimensional and the higher dimensional cases), we
proposed\footnote{~Using a similar analytical technique introduced
in \cite{gm3}, the long-time existence and related asymptotical
behavior of the anisotropic version of the IGCF with zero NBC in
$\mathbb{R}^{n+1}_{1}$ ($n\geq2$) can be expected. As we said in
\cite[Remark 1.1]{gm3}, we left this as an exercise for readers who
are interested in this topic.} in \cite[Remark 1.1]{gm3} that one
can also consider the anisotropic version of the IGCF with zero NBC
in $\mathbb{R}^{n+1}_{1}$, $n\geq2$.\\
 (2) As explained clearly and mentioned in \cite[Remark 1.1]{glm-1},
 we prefer to use L\'{o}pez's setting introduced in \cite{rl} to
 deal with the geometric quantities of spacelike curves in
 $\mathbb{R}^{2}_{1}$ for the purpose of convenience. One can have a
 glance at this setting through the computation of curvature of spacelike graphic curves in
 $\mathbb{R}^{2}_{1}$ (defined over $M\subset\mathscr{H}^{1}(1)$) shown in the proof of Lemma \ref{lemma2-1}
 below.
 \\
 (3) In fact, in Theorem \ref{main1.1}, $M$ is some $convex$ curve segment of the spacelike curve
$\mathscr{H}^{1}(1)$ implies that the curvature of $M$ is positive
everywhere w.r.t. the vector field $\mu$ (provided
its direction is suitably chosen). \\
 (4) Of course, the notation $T_{x}M$ means the tangent
 space at the point $x\in \partial M$, which in the situation here is
 a  $1$-dimensional vector space diffeomorphic to the Euclidean
 $1$-space $\mathbb{R}$. In order to clearly comprehend the property (for the
 vector field $\mu$)
 required in Theorem \ref{main1.1}, we suggest readers to check the
 $3^{rd}$ footnote given in our previous work \cite{GaoY2}, where
 the detailed explanation can be found. \\
 (5)  It is easy to check that
all the arguments in the sequel are still valid for the case
$\alpha=0$ except some minor changes should be made. For instance,
if $\alpha=0$, then the expression (\ref{blow}) below becomes
$\varphi(t)=-t+c$. However, in this setting, one can also get the
$C^0$ estimate as well. Therefore, Theorem \ref{main1.1} also holds
for $\alpha=0$, and in this situation, the homogeneous anisotropic
factor $|X|^{-\alpha}$ equals $1$, and consequently the flow in
Theorem \ref{main1.1} degenerates into the classical IMCF with zero
NBC in $\mathbb{R}^{2}_{1}$. This phenomenon also happens in the
higher dimensional case -- see \cite[Remark 1.1]{gm1} for details. \\
 (6) As shown in Lemma \ref{lemma2-1} below, one can use a
 single parameter $\xi$ to describe any point $x\in
 M\subset\mathscr{H}^{1}(1)$, and moreover, under this
 parametrization, the component of the Riemannian metric on $\mathscr{H}^{1}(1)$ is $\sigma_{\xi\xi} =
g_{\mathscr{H}^{1}(1)}\left(\partial_{\xi} ,
\partial_{\xi}\right)=1$. Without loss of generality, there should
exist an interval $[c,d]$ such that $x(\xi)$, $\xi\in[c,d]$, runs
over the whole curve segment $M$ and $\Sigma\cap\partial M=\{x(c)$,
$x(d)\}$. Furthermore, the lengths $\mathcal{L}(M)$ and
$\mathcal{L}(M_0)$ can be computed  as follows:
\begin{eqnarray*}
\mathcal{L}(M)=\int_{M}d\mathcal{H}^1=\int_{c}^{d}d\xi=d-c,
\end{eqnarray*}
 with $\mathcal{H}^{1}(\cdot)$ the $1$-dimensional Hausdorff measure
 of a prescribed Riemannian curve, and
\begin{eqnarray*}
\mathcal{L}(M_0)=\int_{M_0}d\mathcal{H}^1=\int_{c}^{d}
\sqrt{u_{0}^{2}(\xi)-|Du_{0}(\xi)|^{2}}d\xi,
\end{eqnarray*}
 where as in Lemma \ref{lemma2-1},  $D$ is the covariant connection on $\mathscr{H}^{1}(1)$.
Then the estimate for the radius $r_{\infty}$ in (iii) of Theorem
\ref{main1.1} becomes
\begin{eqnarray*}
\frac{1}{\mathop{{\rm sup}}\limits_{M} u_{0}}\cdot\frac{\int_{c}^{d}
\sqrt{u_{0}^{2}(\xi)-|Du_{0}(\xi)|^{2}}d\xi}{d-c} \leq r_{\infty}
\leq \frac{1}{\mathop{{\rm inf}}\limits_{M}
u_{0}}\cdot\frac{\int_{c}^{d}
\sqrt{u_{0}^{2}(\xi)-|Du_{0}(\xi)|^{2}}d\xi}{d-c}.
\end{eqnarray*}
BTW, the formula for the component of the induced metric on $M_0$
(see (ii) of Lemma \ref{lemma2-1}) and the length formula of curves
have been used directly in the computation of $\mathcal{L}(M)$ and
$\mathcal{L}(M_0)$.
 }
\end{remark}

This paper is organized as follows. In Section \ref{se2}, we will
prove the short-time existence of the flow discussed in Theorem
\ref{main1.1} (i.e., the IMCF with zero NBC in
$\mathbb{R}^{2}_{1}$). In Section \ref{se3}, several estimates,
including $C^0$, time-derivative and gradient estimates, of
solutions to the flow equation will be shown in details. Estimates
of higher-order derivatives of solutions to the flow equation, which
naturally leads to the long-time existence of the flow, will be
investigated in Section \ref{se4}. In the end, we will clearly show
the convergence of the rescaled flow in Section \ref{se6}.

\section{The scalar version of the flow equation} \label{se2}

Since the spacelike $C^{2,\gamma}$-curve $M_{0}$ can be written as a
graph of $M \subset \mathscr{H}^{1}(1)$, there exists a function
$u_{0} \in C^{2,\gamma} (M)$ such that $X_{0} : M \rightarrow
\mathbb{R}^{2}_{1}$ has the form $x\mapsto G_{0} := (x,u_{0}(x)).$
The curve $M_{t}$ given by the embedding
\begin{eqnarray*}
X(\cdot , t) : M \rightarrow \mathbb{R}^{2}_{1}
\end{eqnarray*}
at time $t$ may be represented as a graph over $M \subset
\mathscr{H}^{1}(1)$, and then we can make ansatz
\begin{eqnarray*}
 X(x , t) = \left(x,u(x,t)\right)
\end{eqnarray*}
for some function $u : M \times [0,T) \rightarrow \mathbb{R}.$ The
following formulae are needed.

\begin{lemma}\label{lemma2-1}
Define $p := X(x,t)$ and assume that a point on $\mathscr{H}^{1}(1)$
is parameterized by the coordinate $\xi$, that is, $x=x(\xi)$. By
the abuse of notations, let $\partial_{\xi}$ be the corresponding
coordinate field on $\mathscr{H}^{1}(1)$ and $\sigma_{\xi\xi} =
g_{\mathscr{H}^{1}(1)}\left(\partial_{\xi} ,
\partial_{\xi}\right)=1$ be the Riemannian metric on
$\mathscr{H}^{1}(1)$. Denote by $u_{\xi}:=D_{\partial_{\xi}}u,$ and
$u_{\xi\xi}:=D_{\partial_{\xi}}D_{\partial_{\xi}}u$ the covariant
derivatives of u w.r.t. the metric $g_{\mathscr{H}^{1}(1)},$ where D
is the covariant connection on $\mathscr{H}^{1}(1)$. Let $\nabla$ be
the Levi-Civita connection of $M_{t}$ w.r.t. the metric
$g:=u^{2}g_{\mathscr{H}^{1}(1)} - dr^{2}$ induced from the
Lorentzian metric $\langle\cdot,\cdot\rangle_{L}$ of
$\mathbb{R}^{2}_{1}$. The following formulae hold:

(i) The tangential vector on $M_{t}$ is
\begin{eqnarray*}
X_{\xi} = \partial _{\xi} + u_{\xi}\partial _{r},
\end{eqnarray*}
and the corresponding past-directed timelike unit normal vector is
given by
\begin{eqnarray*}
\nu = \displaystyle{-\frac{1}{v}}\displaystyle{\left(
\frac{u_{\xi}}{u^{2}}\partial_{\xi} + \partial_{r}\right)},
\end{eqnarray*}
where $u^{\xi}=\sigma^{\xi\xi} u_{\xi}=u_{\xi}$,
$|Du|^{2}=u_{\xi}u^{\xi}=|u_{\xi}|^{2}$, and $v=\sqrt{1 -
u^{-2}\left|Du\right|^{2}}.$

(ii) The induced metric g on $M_{t}$ has the form
\begin{eqnarray*}
g_{\xi\xi} = u^{2}\sigma_{\xi\xi}-u_{\xi}^{2} = u^{2}-u_{\xi}^{2},
\end{eqnarray*}
and its inverse is given by
\begin{eqnarray*}
g^{\xi\xi} =
\frac{1}{u^{2}}\left(\sigma^{\xi\xi}+\frac{u_{\xi}^{2}}{u^{2}v^{2}}\right)
= \frac{1}{u^{2}}\left(1+\frac{u_{\xi}^{2}}{u^{2}v^{2}}\right).
\end{eqnarray*}

(iii) The curvature of $M_{t}$ is given by
\begin{eqnarray*}
k=\frac{u_{\xi\xi}u + u^{2} -2u^{2}_{\xi}}{v^{3}u^{3}}.
\end{eqnarray*}

(iv) Let $p = X(x , t)\in \Sigma $ with $x \in \partial M$,
$\hat\mu(p) \in T_{p}M_{t}$, $\mu = \mu ^{\xi}(x)\partial_{\xi}(x)$
at $x$, with $\partial_{\xi}$ the basis vector of $T_{x}M$. Then
\begin{eqnarray*}
\langle \hat\mu(p) ,\nu(p) \rangle_{L} = 0 \Leftrightarrow \mu
^{\xi}(x)u_{\xi}(x,t) = 0.
\end{eqnarray*}
\end{lemma}

\begin{proof}
The proof (except the calculation of curvature of spacelike graphic
curves in $\mathbb{R}^{2}_{1}$) is very similar to that of
\cite[Lemma 3.1]{GaoY2}, and we prefer to omit this part.

As mentioned in (2) of Remark \ref{remark1.1}, now, we prefer to use
L\'{o}pez's setting introduced in \cite{rl} to compute the
curvature. Clearly, the tangential vector on $M_{t}$ can be
rewritten as
\begin{eqnarray*}
X_{\xi}=(1,u_{\xi}),
\end{eqnarray*}
the Minkowski arc-length parameter is defined as
\begin{eqnarray*}
ds = \sqrt{|\langle X_{\xi} ,X_{\xi}\rangle_{L}|}d\xi,
\end{eqnarray*}
and the \emph{unit} tangent vector is
\begin{eqnarray*}
T= X_{s}:=\frac{1}{\sqrt{|\langle X_{\xi}
,X_{\xi}\rangle_{L}|}}X_{\xi}.
\end{eqnarray*}
Hence, we have
\begin{eqnarray*}
T:= (T^{1} , T^{2})=\frac{1}{\sqrt{u^{2}-u_{\xi}^{2}}}(1,u_{\xi}),
\end{eqnarray*}
which implies\footnote{~As usual, the comma ``," in subscript of a
given tensor means doing covariant derivatives. We also make an
agreement that, for simplicity, in the sequel the comma ``," in
subscripts will be omitted unless necessary. }
\begin{eqnarray*}
\nabla_{X_{\xi}}T=T_{,X_{\xi}}:=T_{\xi} = T^{1}_{,\xi}\partial_{\xi}
+ T^{2}_{,\xi}\partial_{r},
\end{eqnarray*}
so we get
\begin{eqnarray*}
T_{\xi} =
\left(\frac{uu_{\xi}u_{\xi\xi}+u_{\xi}u^{2}-2u_{\xi}^{3}}{u(u^{2}-u_{\xi}^{2})^{\frac{3}{2}}}
,
\frac{u^{2}u_{\xi\xi}+u^{3}-2u_{\xi}^{2}u}{(u^{2}-u_{\xi}^{2})^{\frac{3}{2}}}\right),
\end{eqnarray*}
and
\begin{eqnarray} \label{add-xx}
T_{s} = T_{\xi}\frac{d\xi}{ds} =
\left(\frac{uu_{\xi}u_{\xi\xi}+u_{\xi}u^{2}-2u_{\xi}^{3}}{u(u^{2}-u_{\xi}^{2})^{2}}
,
\frac{u^{2}u_{\xi\xi}+u^{3}-2u_{\xi}^{2}u}{(u^{2}-u_{\xi}^{2})^{2}}\right).
\end{eqnarray}
Since $T$ is spacelike, the curvature $k$ can be defined as (see
\cite[pp. 14-16]{rl})
\begin{eqnarray*}
k=|T_{s}|=\sqrt{|\langle T_{s},T_{s}\rangle_{L}|},
\end{eqnarray*}
which, together with (\ref{add-xx}), gives
\begin{eqnarray*}
k=\frac{\left|uu_{\xi\xi}+u^{2}-2u_{\xi}^{2}\right|}{u^{3}v^{3}} .
\end{eqnarray*}
Without loss of generality\footnote{~ Clearly, $k\neq0$, which
implies $u_{\xi\xi}u+u^{2}-2u_{\xi}^{2} \neq 0$, otherwise, the IMCF
equation in (\ref{Eq1}) would degenerate. If
$u_{\xi\xi}u+u^{2}-2u_{\xi}^{2} < 0$, one has
$k=\frac{2u_{\xi}^{2}-u_{\xi\xi}u-u^{2}}{u^{3}v^{3}}$, and then one
can choose the future-directed timelike unit normal vector $N$ such
that the flow equation becomes $\frac{\partial X}{\partial
t}=\frac{1}{|X|^{\alpha}k}N$, which is still parabolic. Then similar
argument can be made to get the main conclusions in Theorem
\ref{main1.1}. }, assume that $u_{\xi\xi}+u^{2}-2u_{\xi}^{2}
>0 $, and then
\begin{eqnarray*}
k=\left|T_{s}\right|=\frac{uu_{\xi\xi}+u^{2}-2u_{\xi}^{2}}{u^{3}v^{3}},
\end{eqnarray*}
Choosing the unit normal $N$ as follows
\begin{eqnarray*}
N=\frac{u}{\sqrt{u^{2}-u_{\xi}^{2}}}\left(\frac{u_{\xi}}{u^{2}},
1\right),
\end{eqnarray*}
which is future-directed, and clearly it satisfies
\begin{eqnarray*}
T_{s}=kN.
\end{eqnarray*}
The above equation is actually the Frenet formula for spacelike
curves $M_{t}$ in $\mathbb{R}^{2}_{1}$. The proof is finished.
\end{proof}

Using techniques as in Ecker \cite{Eck} (see also \cite{Ge90, Ge06,
Mar}), the problem \eqref{Eq1} can be reduced to solve the following
scalar equation with the corresponding initial data and the
corresponding NBC
\begin{equation}\label{Eq2}
\left\{
\begin{aligned}
&\frac{\partial u}{\partial t}=-\frac{v}{u^{\alpha}k} \qquad
&&~\mathrm{in}~
M\times(0,\infty)\\
&\nabla_{\mu}u=0  \qquad&&~\mathrm{on}~ \partial M\times(0,\infty)\\
&u(\cdot,0)=u_{0} \qquad &&~\mathrm{in}~M.
\end{aligned}
\right.
\end{equation}
By Lemma \ref{lemma2-1}, define a new function $\varphi(x,t) = \ln
u(x,t)$ and then the curvature can be rewritten as\footnote{~One
might find that our previous works on inverse curvature flows (see,
e.g., \cite{cmtw,glm-1,GaoY2,gm1,gm3,chmw}) have used $\varphi=\log
u$ to represent the logarithmic relation between functions $\varphi$
and $u$. However, we did not prescribe a base for the logarithmic
function $\log u$. This is because of two reasons. First, in most of
non-Chinese mathematical literatures, there is no strict difference
between $y=\log x$ and $e$-base logarithmic function $y=\ln x$, that
is, in those literatures, $y=\log x$ was treated as $y=\ln x$.
Second, in nearly all the analysis process, the non $e$-base
logarithmic function has almost the same function with the $e$-base
logarithmic function. Based on these two reasons, there is not
necessary to give a base for the logarithmic function when doing
changes of variable. Hence, by abuse of notations, we will use
$\log$, $\ln$ simultaneously, and will give a prescribed base
clearly if necessary.
 }
\begin{eqnarray*}
k=\frac{e^{-\varphi}}{v}\left(1+\frac{1}{v^{2}}\varphi_{\xi\xi}\right).
\end{eqnarray*}
Hence, the evolution equation in \eqref{Eq2} can be rewritten as
\begin{eqnarray*}
\frac{\partial}{\partial t}\varphi = -e^{-\alpha \varphi}\left(1 -
\left|D\varphi\right|^{2}\right)\frac{1}{\left[1+\frac{1}{v^{2}}\varphi_{\xi\xi}\right]}
:=Q\left(\varphi ,D\varphi, D^{2}\varphi\right).
\end{eqnarray*}
Thus, the problem (1.1) is again reduced to solve the following
scalar equation with the NBC and the initial data
\begin{equation}\label{Evo-1}
\left\{
\begin{aligned}
&\frac{\partial \varphi}{\partial t}=Q\left(\varphi ,D\varphi,
D^{2}\varphi\right) \qquad &&~\mathrm{in}~
M\times(0,T)\\
&\nabla_{\mu}\varphi=0  \qquad&&~\mathrm{on}~ \partial M\times(0,T)\\
&\varphi(\cdot,0)=\varphi_{0} \qquad &&~\mathrm{in}~M,
\end{aligned}
\right.
\end{equation}
where
\begin{eqnarray*}
\left(1+\frac{1}{v^{2}}\varphi_{0,\xi\xi}\right)
\end{eqnarray*}
is positive on $M$, since $M_{0}$ is convex. Clearly, for the
initial spacelike graphic curve $M_{0}$,
\begin{eqnarray*}
\left.\frac{\partial Q}{\partial \varphi_{\xi\xi}}\right|_{\varphi
_{0}} = \frac{1}{u^{2+\alpha}k^{2}v^{2}}
\end{eqnarray*}
is positive on $M$. Based on the above facts, as in \cite{Ge90,
Ge06, Mar}, we can get the following short-time existence and
uniqueness for the parabolic system \eqref{Eq1}.

\begin{lemma}\label{Lemma2-2}
Let $X_{0}(M) = M_{0}$ be as in Theorem \ref{main1.1}. Then there
exist some $T > 0$, a unique solution $u \in C^{2+\gamma ,
1+\frac{\gamma}{2}}\left(M\times \left[0,T\right]\right) \cap
C^{\infty}\left(M \times \left(0,T\right]\right)$, where $\varphi
(x,t) = \log u(x,t)$, to the parabolic system \eqref{Evo-1} with the
coefficient
\begin{eqnarray*}
\left(1+\frac{1}{v^{2}}\varphi_{\xi\xi}\right)
\end{eqnarray*}
positive on $M$. Thus there exists a unique map $\psi :M \times
[0,T] \rightarrow M$ such that $\psi(\partial M , t) = \partial M$
and the map $\widehat{X}$ defined by
\begin{eqnarray*}
\widehat{X} : M \times [0,T) \rightarrow \mathbb{R}^{2}_{1} :(x,t)
\mapsto X(\psi (x,t),t)
\end{eqnarray*}
has the same regularity as stated in Theorem \ref{main1.1} and is
the unique solution to the parabolic system \eqref{Eq1}.
\end{lemma}

Let $T^{\ast}$ be the maximal time such that there exists some
\begin{eqnarray*}
u\in C^{2+\gamma , 1+\frac{\gamma}{2}}\left(M \times
\left[0,T^{\ast}\right)\right) \cap C^{\infty}\left(M \times
\left(0,T^{\ast}\right)\right)
\end{eqnarray*}
which solves \eqref{Evo-1}. In the sequel, we shall prove a priori
estimates for those admissible solutions on $[0,T]$ where
$T<T^{\ast}$.

\section{ $C^{0}$, $\dot{\varphi}$ and gradient estimates} \label{se3}

\begin{lemma}[\textbf{$C^{0}$ estimate}]\label{estimate1}
Let $\varphi$ be a solution of \eqref{Evo-1}, and then for
$\alpha<0$, we have
\begin{eqnarray*}
c_{1}\leq u(x,t)\Theta^{-1}(t,c)\leq c_{2}, \qquad\quad\forall
~~x\in M, ~~t\in [0,T]
\end{eqnarray*}
for some positive constants $c_{1},$ $c_{2},$ where
$\Theta(t,c):=\{-\alpha t + e^{\alpha c}\}^{\frac{1}{\alpha}}$ with
\begin{eqnarray*}
\mathop{{\rm inf}}\limits_{M}\hspace{0.15em}\varphi (\cdot,0)\leq
c\leq \mathop{{\rm sup}}\limits_{M}\hspace{0.15em}\varphi(\cdot,0)
\end{eqnarray*}
\end{lemma}
\begin{proof}
Let $\varphi(x,t) = \varphi(t)$ (independent of $x$) be the solution
of \eqref{Evo-1} with $\varphi(0) = c$. In this case, the first
equation in \eqref{Evo-1} reduces to an ODE
\begin{eqnarray*}
\frac{d}{dt}\varphi=-e^{-\alpha \varphi}.
\end{eqnarray*}
Therefore,
\begin{eqnarray}\label{blow}
\varphi(t)=\frac{1}{\alpha}\ln (-\alpha t + e^{\alpha c}),\qquad{\rm
for}~~\alpha<0.
\end{eqnarray}
Using the maximum principle, we can obtain that
\begin{eqnarray}\label{C0}
\frac{1}{\alpha}\ln (-\alpha t+e^{\alpha\varphi_{1}})\leq
\varphi(x,t)\leq \frac{1}{\alpha}\ln (-\alpha
t+e^{\alpha\varphi_{2}}),
\end{eqnarray}
where $\varphi_{1}:={\rm inf}_{M}\varphi (\cdot,0)$ and
$\varphi_{2}:={\rm sup}_{M}\varphi (\cdot,0)$. The estimate is
obtained since $\varphi = \ln u$.
\end{proof}

\begin{lemma}[\bf$\dot{\varphi}$ estimate]\label{estimate2}
Let $\varphi$ be a solution of \eqref{Evo-1} and $\Sigma$ be the
boundary of a smooth, convex cone defined as in Theorem
\ref{main1.1}, then for $\alpha<0$,
\begin{eqnarray*}
{\rm min}\left\{\mathop{{\rm inf}}\limits_{M}(\dot{\varphi}(\cdot
,0)\cdot \Theta(0)^{\alpha}),-1\right\}\leq
\dot{\varphi}(x,t)\Theta(t)^{\alpha}\leq {\rm
max}\left\{\mathop{{\rm sup}}\limits_{M}(\dot{\varphi}(\cdot
,0)\cdot \Theta(0)^{\alpha}),-1\right\}.
\end{eqnarray*}
\end{lemma}
\begin{proof}
Set
\begin{eqnarray*}
\mathcal{M}(x(\xi),t)=\dot{\varphi}(x(\xi), t)\Theta(t)^{\alpha}.
\end{eqnarray*}
Differentiating both sides of the first evolution equation of
\eqref{Evo-1}, it is easy to get that
\begin{eqnarray*}\label{Eq-1}
\left\{
\begin{aligned}
&\frac{\partial \mathcal{M}}{\partial
t}=Q^{\xi\xi}\mathcal{M}_{\xi\xi}+Q^{\xi}\mathcal{M}_{\xi}-\alpha\Theta^{-\alpha}\left(1+\mathcal{M}\right)\mathcal{M}
\qquad &&~\mathrm{in}~
M\times(0,T)\\
&\nabla_{\mu}\mathcal{M}=0  \qquad&&~\mathrm{on}~ \partial M\times(0,T)\\
&\mathcal{M}(\cdot,0)=\dot\varphi_{0}\cdot\Theta(0)^{\alpha} \qquad
&&~\mathrm{in}~M,
\end{aligned}
\right.
\end{eqnarray*}
where $Q^{\xi\xi}:=\frac{\partial Q}{\partial \varphi_{\xi\xi}}$ and
$Q^{\xi}:=\frac{\partial Q}{\partial \varphi_{\xi}}.$ Then the
result follows from the maximum principle.
\end{proof}

\begin{lemma}[\textbf{Gradient estimate}]\label{estimate3}
Let $\varphi$ be a solution of \eqref{Evo-1} and $\Sigma$ be the
boundary of a smooth, convex cone described as in Theorem
\ref{main1.1}. Then for $\alpha<0$, we have
\begin{eqnarray*}\label{gradient}
\left|D\varphi\right|\leq \mathop{{\rm
sup}}\limits_{M}\left|D\varphi(\cdot,0)\right|<1,\qquad\forall~~
x\in M,~~ t\in\left[0,T\right].
\end{eqnarray*}
\end{lemma}
\begin{proof}
Set $\psi=\frac{\left|D\varphi\right|^{2}}{2}$. By differentiating
$\psi$, we have
\begin{eqnarray*}
\frac{\partial\psi}{\partial t}=\frac{\partial}{\partial
t}\varphi_{\xi}\cdot\varphi_{\xi}=Q_{\xi}\varphi_{\xi}.
\end{eqnarray*}
Then using the evolution equation of $\varphi$ in (\ref{Evo-1})
yields
\begin{eqnarray*}
\frac{\partial\psi}{\partial
t}=Q^{\xi\xi}\varphi_{\xi\xi\xi}\varphi_{\xi}+Q^{\xi}\varphi_{\xi\xi}\varphi_{\xi}-\alpha
Q\varphi_{\xi}^{2}.
\end{eqnarray*}
Therefore, we can express $\varphi_{\xi\xi\xi}\varphi_{\xi}$ as
\begin{eqnarray*}
\varphi_{\xi\xi\xi}\varphi_{\xi}=\psi_{\xi\xi}-\varphi_{\xi\xi}^{2}.
\end{eqnarray*}
Also, we can express $\varphi_{\xi\xi}\varphi_{\xi}$ as
\begin{eqnarray*}
\varphi_{\xi\xi}\varphi_{\xi}=\psi_{\xi}.
\end{eqnarray*}
Then, we have
\begin{eqnarray}\label{gra}
\frac{\partial\psi}{\partial
t}=Q^{\xi\xi}\psi_{\xi\xi}+Q^{\xi}\psi_{\xi}-Q^{\xi\xi}\varphi_{\xi\xi}^{2}-\alpha
Q\varphi_{\xi}^{2}
\end{eqnarray}

Since $Q^{\xi\xi}$ and $\alpha Q$ is positive definite, the third
and fourth terms in the RHS of \eqref{gra} are non-positive. Since
$\varphi$ is a solution of \eqref{Evo-1}, that is,
\begin{eqnarray*}
\nabla_{\mu}\varphi=\mu^{\xi}\varphi_{\xi}=0\qquad {\rm on}~~
\partial M\times(0,T),
\end{eqnarray*}
where $\mu=\mu^{\xi}\partial_{\xi}$, we have
\begin{eqnarray*}
\nabla_{\mu}\psi=\mu^{\xi}\psi_{\xi}=\mu^{\xi}\varphi_{\xi\xi}\varphi_{\xi}
\qquad {\rm on}~~ \partial M\times(0,T).
\end{eqnarray*}
So, we can get
\begin{eqnarray}\label{Eq-2}
\left\{
\begin{aligned}
&\frac{\partial\psi}{\partial t}\leq Q^{\xi\xi}\psi_{\xi\xi}+Q^{\xi}\psi_{\xi} \qquad &&~\mathrm{in}~M\times(0,T)\\
&\nabla_{\mu}\psi=0  \qquad&&~\mathrm{on}~ \partial M\times(0,T)\\
&\psi(\cdot,0)=\frac{\left|\varphi_{\xi}(\cdot,0)\right|^{2}}{2}
\qquad &&~\mathrm{in}~M.
\end{aligned}
\right.
\end{eqnarray}
Using the maximum principle, we have
\begin{eqnarray*}
\left|D\varphi\right|\leq \mathop{{\rm
sup}}\limits_{M}\left|D\varphi(\cdot,0)\right|.
\end{eqnarray*}
Since $G_{0}=\{\left(x(\xi),u(x(\xi),0)\right)|x\in M\}$ is a
spacelike graph of $\mathbb{R}^{2}_{1}$, so we have
\begin{eqnarray*}
\left|D\varphi\right|\leq \mathop{{\rm
sup}}\limits_{M}\left|D\varphi(\cdot,0)\right|<1,\qquad\forall~~
x\in M,~~ t\in\left[0,T\right].
\end{eqnarray*}
Our proof is finished.
\end{proof}

\begin{remark}
\rm{The gradient estimate in Lemma \ref{estimate3} makes sure that
the evolving graphs $G_{t}:=\{\left(x,u(x,t)\right)|x\in M,0\leq
t\leq T\}$ are spacelike graphs.}
\end{remark}

Combing the gradient estimate with $\dot\varphi$ estimate, we can
obtain:

\begin{corollary}
If $\varphi$ satisfies \eqref{Evo-1}, then we have
\begin{eqnarray*}\label{k-boundary}
0<c_{3}\leq k\Theta\leq c_{4}<+\infty
\end{eqnarray*}
where $c_{3}$ and $c_{4}$ are positive constants independent of
$\varphi$.
\end{corollary}

\section{H\"{o}lder estimates and the long-time existence} \label{se4}

Set $\Phi=\frac{1}{\left|X\right|^{\alpha}k}$,
$w=\left<X,\nu\right>_{L}$ and $\Psi=\frac{\Phi}{w}$. We can get the
following evolution equations:

\begin{lemma}\label{lemma4.1}
Under the assumptions of Theorem \ref{main1.1}, we have
\begin{eqnarray*}
\frac{\partial}{\partial t}g_{\xi\xi}=-2\Phi kg_{\xi\xi},
\end{eqnarray*}
\begin{eqnarray*}
\frac{\partial}{\partial t}g^{\xi\xi}=2\Phi kg^{\xi\xi},
\end{eqnarray*}
\begin{eqnarray*}
\frac{\partial}{\partial t}\nu=\nabla\Phi X_{\xi},
\end{eqnarray*}
\begin{eqnarray*}
\begin{split}
\partial_{t}k=\Phi k^{2} - \alpha(\alpha+1)\Phi u^{-2}|\nabla u|^{2} + \alpha u^{-1}\Phi\Delta u - 2\Phi k^{-2}|\nabla k|^{2} + \Phi k^{-1}\Delta k
\end{split}
\end{eqnarray*}
and
\begin{eqnarray}\label{div-1}
\begin{split}
\frac{\partial\Psi}{\partial t}&=\mathrm{div}_{g}\left(u^{-\alpha}k^{-2}\nabla\Psi\right) - 2k^{-2}u^{-\alpha}\Psi^{-1}|\nabla\Psi|^{2} \\
&+ \alpha\Psi^{2} - \alpha
u^{-\alpha-1}k^{-2}u_{\xi}\nabla^{\xi}\Psi +
\alpha\Psi^{2}u^{-1}\nabla^{\xi}u\left<X,X_{\xi}\right>_{L}.
\end{split}
\end{eqnarray}
\end{lemma}

\begin{proof}
It is easy to get the first three evolution equations, and we omit
here. Direct calculation results in
\begin{eqnarray*}
\Phi_{\xi}=-\alpha \Phi u^{-1}u_{\xi}-\Phi k^{-1}k_{\xi},
\end{eqnarray*}
and
\begin{eqnarray*}
\Phi_{\xi\xi}=\alpha(\alpha+1)\Phi u^{-2}u_{\xi}^{2}+2\alpha
u^{-1}\Phi k^{-1}u_{\xi}k_{\xi}-\alpha u^{-1}\Phi u_{\xi\xi}-\Phi
k^{-1}k_{\xi\xi}+2\Phi k^{-2}k_{\xi\xi}^{2}.
\end{eqnarray*}
According to the definition of curvature, we have
\begin{eqnarray*}
k=g^{\xi\xi}\left<X_{\xi\xi},\nu\right>_{L},
\end{eqnarray*}
and then
\begin{eqnarray*}
\partial_{t}k=\frac{\partial}{\partial t}g^{\xi\xi}\cdot\left<X_{\xi\xi},\nu\right>_{L}+g^{\xi\xi}\cdot\frac{\partial}{\partial
t}\left<X_{\xi\xi},\nu\right>_{L}.
\end{eqnarray*}
Since
\begin{eqnarray*}
\frac{\partial}{\partial
t}\left<X_{\xi\xi},\nu\right>_{L}=-\Phi_{\xi\xi}-\Phi
k^{2}g_{\xi\xi},
\end{eqnarray*}
so
\begin{eqnarray*}
\partial_{t}k=-\Delta\Phi + \Phi k^{2}.
\end{eqnarray*}
Thus,
\begin{eqnarray*}
\begin{split}
\partial_{t}k&=- \alpha(\alpha+1)\Phi u^{-2}|\nabla u|^{2}+\alpha u^{-1}\Phi\Delta u + \Phi k^{2} \\
&- 2\Phi k^{-2}|\nabla k|^{2}- 2\alpha u^{-1}\Phi
k^{-1}u_{\xi}\nabla^{\xi}k + \Phi k^{-1}\Delta k .
\end{split}
\end{eqnarray*}
Clearly,
\begin{eqnarray*}
\partial_{t}w=-\Phi-\alpha\Phi u^{-1}\nabla ^{\xi}u\left<X,X_{\xi}\right>_{L}-\Phi k^{-1}\nabla ^{\xi}k\left<X,X_{\xi}\right>_{L},
\end{eqnarray*}
by the calculation, we have
\begin{eqnarray*}
w_{\xi}=-k\left<X,X_{\xi}\right>_{L},
\end{eqnarray*}
\begin{eqnarray*}
w_{\xi\xi}=-kg_{\xi\xi} -k_{\xi}\left<X,X_{\xi}\right>_{L} +
k^{2}g_{\xi\xi}w.
\end{eqnarray*}
Thus,
\begin{eqnarray*}
\Delta w=-k -\nabla^{\xi}k\left<X,X_{\xi}\right>_{L} + k^{2}w,
\end{eqnarray*}
and
\begin{eqnarray*}
\partial_{t} w= -u^{-\alpha}w -\alpha u^{-\alpha-1}k^{-1}\nabla^{\xi}u\left<X,X_{\xi}\right>_{L} + u^{-\alpha}k^{-2}\Delta w.
\end{eqnarray*}
Hence
\begin{eqnarray*}
\begin{split}
\frac{\partial\Psi}{\partial t}&=\alpha\frac{1}{u^{1+\alpha}}\frac{1}{kw}\frac{1}{u^{\alpha-1}kw}-\frac{1}{u^{\alpha}k^{2}}\frac{1}{w}\partial_{t}k -\frac{1}{u^{\alpha}k}\frac{1}{w^{2}}\partial_{t}w\\
&=\alpha u^{-2\alpha}k^{-2}w^{-2} + \alpha(\alpha+1)u^{-2\alpha-2}k^{-3}w^{-1}|\nabla u|^{2} + 2u^{-2\alpha}k^{-5}w^{-1}|\nabla k|^{2}\\
&+ 2\alpha u^{-2\alpha-1}k^{-4}w^{-1}u_{\xi}\nabla^{\xi}k - \alpha u^{-2\alpha-1}k^{-3}w^{-1}\Delta u - u^{-2\alpha}k^{-4}w^{-1}\Delta k\\
&- u^{-2\alpha}k^{-3}w^{-2}\Delta w +\alpha
u^{-2\alpha-1}k^{-2}w^{-2}\nabla^{\xi}u\left<X,X_{\xi}\right>_{L}.
\end{split}
\end{eqnarray*}
In order to prove \eqref{div-1}, we calculate
\begin{eqnarray*}
\Psi_{\xi}=-\alpha u^{-\alpha-1}k^{-1}w^{-1}u_{\xi} -
u^{-\alpha}k^{-2}w^{-1}k_{\xi}- u^{-\alpha}k^{-1}w^{-2}w_{\xi}.
\end{eqnarray*}
and
\begin{eqnarray*}
\begin{split}
\Psi_{\xi\xi}&=\alpha(\alpha+1) u^{-\alpha-2}k^{-1}w^{-1}u_{\xi}^{2} + 2u^{-\alpha}k^{-3}w^{-1}k_{\xi}^{2} + 2u^{-\alpha}k^{-1}w^{-3}w_{\xi}^{2}\\
&+2\alpha u^{-\alpha-1}k^{-2}w^{-1}u_{\xi}k_{\xi} + 2u^{-\alpha}k^{-2}w^{-2}k_{\xi}w_{\xi} + 2\alpha u^{-\alpha-1}k^{-1}w^{-2}w_{\xi}u_{\xi}\\
&-\alpha u^{-\alpha-1}k^{-1}w^{-1}u_{xx} -
u^{-\alpha}k^{-2}w^{-1}k_{\xi\xi} -
u^{-\alpha}k^{-1}w^{-2}w_{\xi\xi}.
\end{split}
\end{eqnarray*}
Therefore,
\begin{eqnarray*}
\begin{split}
u^{-\alpha}k^{-2}\Delta\Psi&=\alpha(\alpha+1) u^{-2\alpha-2}k^{-3}w^{-1}|\nabla u|^{2} + 2u^{-2\alpha}k^{-5}w^{-1}|\nabla k|^{2} + 2u^{-2\alpha}k^{-3}w^{-3}|\nabla w|^{2}\\
&+2\alpha u^{-2\alpha-1}k^{-4}w^{-1}u_{\xi}\nabla^{\xi}k + 2u^{-2\alpha}k^{-4}w^{-2}w_{\xi}\nabla^{\xi}k + 2\alpha u^{-2\alpha-1}k^{-3}w^{-2}w_{\xi}\nabla^{\xi}u\\
&-\alpha u^{-2\alpha-1}k^{-3}w^{-1}\Delta u -
u^{-2\alpha}k^{-4}w^{-1}\Delta k - u^{-2\alpha}k^{-3}w^{-2}\Delta w.
\end{split}
\end{eqnarray*}
So we have
\begin{eqnarray*}
\begin{split}
{\rm div}(u^{-\alpha}k^{-2}\nabla\Psi) &= \alpha(2\alpha+1)u^{-2\alpha-2}k^{-3}w^{-1}|\nabla u|^{2} + 4u^{-2\alpha}k^{-5}w^{-1}|\nabla k|^{2} + 2u^{-2\alpha}k^{-3}w^{-3}|\nabla w|^{2}\\
&+ 5\alpha u^{-2\alpha-1}k^{-4}w^{-1}u_{\xi}\nabla^{\xi}k +4u^{-2\alpha}k^{-4}w^{-2}w_{\xi}\nabla^{\xi}k + 3\alpha u^{-2\alpha-1}k^{-3}w^{-2}w_{\xi}\nabla^{\xi}u\\
&- \alpha u^{-2\alpha-1} k^{-3}w^{-1}\Delta u - u^{-2\alpha}
k^{-4}w^{-1}\Delta k - u^{-2\alpha} k^{-3}w^{-2}\Delta w.
\end{split}
\end{eqnarray*}
and
\begin{eqnarray*}
\begin{split}
2k^{-1}w|\nabla\Psi|^{2} &= 2\alpha^{2}u^{-2\alpha-2}k^{-3}w^{-1}|\nabla u|^{2} + 2u^{-2\alpha}k^{-5}w^{-1}|\nabla k|^{2} + 2u^{-2\alpha}k^{-3}w^{-3}|\nabla w|^{2}\\
&+ 4\alpha u^{-2\alpha-1}k^{-4}w^{-1}u_{\xi}\nabla^{\xi}k
+4u^{-2\alpha}k^{-4}w^{-2}w_{\xi}\nabla^{\xi}k + 4\alpha
u^{-2\alpha-1}k^{-3}w^{-2}w_{\xi}\nabla^{\xi}u.
\end{split}
\end{eqnarray*}
In sum, we have
\begin{eqnarray*}
\begin{split}
&\frac{\partial\Psi}{\partial t} - \mathrm{div}_{g}\left(u^{-\alpha}k^{-2}\nabla\Psi\right) + 2k^{-1}w|\nabla\Psi|^{2}\\
&= \alpha u^{-2\alpha}k^{-2}w^{-2} + \alpha u^{-2\alpha-1}k^{-4}w^{-1}u_{\xi}\nabla^{\xi}k + \alpha u^{-2\alpha-1}k^{-3}w^{-2}w_{\xi}\nabla^{\xi}u\\
&\qquad +\alpha^{2}u^{-2\alpha-2}k^{-3}w^{-1}|\nabla u|^{2} + \alpha u^{-2\alpha-1}k^{-2}w^{-2}\nabla^{\xi}u\left<X,X_{\xi}\right>_{L}\\
&= \alpha\Psi^{2} - \alpha
u^{-\alpha-1}k^{-2}u_{\xi}\nabla^{\xi}\Psi +
\alpha\Psi^{2}u^{-1}\nabla^{\xi}u\left<X,X_{\xi}\right>_{L}.
\end{split}
\end{eqnarray*}
The proof is finished.
\end{proof}

Now, we define the rescaled flow by
\begin{eqnarray*}
\widetilde{X}=X\Theta^{-1}.
\end{eqnarray*}
Thus,
\begin{eqnarray*}
\widetilde{u}=u\Theta^{-1},\\
\widetilde{\varphi}=\varphi-\ln\Theta,
\end{eqnarray*}
and the rescaled mean curvature equation takes the form
\begin{eqnarray*}
\frac{\partial}{\partial t}\widetilde{u} =
-\frac{v}{\widetilde{u}^{\alpha}\widetilde{k}}\Theta^{-\alpha} +
\widetilde{u}\Theta^{-\alpha}.
\end{eqnarray*}
Defining $t=t(s)$ by the relation
\begin{eqnarray*}
\frac{dt}{ds}=\Theta^{\alpha}
\end{eqnarray*}
such that $t(0)=0$ and $t(S)=T$. Then $\widetilde{u}$ satisfies
\begin{eqnarray}\label{Eq-re}
\left\{
\begin{aligned}
&\frac{\partial}{\partial s}\widetilde{u}=-\frac{v}{\widetilde{u}^{\alpha}\widetilde{k}}+\widetilde{u}  \qquad&&~\mathrm{in}~ \partial M\times(0,S)\\
&\nabla_{\mu}\widetilde{u}=0  \qquad&&~\mathrm{on}~ \partial M\times(0,S)\\
&\widetilde{u}(\cdot,0)=\widetilde{u}_{0} \qquad &&~\mathrm{in}~M.
\end{aligned}
\right.
\end{eqnarray}

\begin{lemma}\label{lemma4.2}
Let X be a solution of \eqref{main1.1} and $\widetilde{X} =
X\Theta^{-1}$ be the rescaled solution. Then
\begin{eqnarray*}
\begin{split}
&D\widetilde{u}=Du\Theta^{-1} ,~~~~ D\widetilde{\varphi}=D\varphi , ~~~~\frac{\partial \widetilde{u}}{\partial s}=\frac{\partial u}{\partial t}\Theta^{\alpha-1} + u\Theta^{-1},\\
&\widetilde{g}_{\xi\xi}=\Theta^{-2}g_{\xi\xi},~~~~\widetilde{g}^{\xi\xi}=\Theta^{2}g^{\xi\xi},~~~~\widetilde{k}=k\Theta.
\end{split}
\end{eqnarray*}
\end{lemma}
\begin{proof}
These relations can be computed directly.
\end{proof}

\begin{lemma}\label{lemma4.3}
Let u be a solution to the parabolic system \eqref{Evo-1}, where
$\varphi(x,t)=\ln u(x,t)$, and $\Sigma$ be the boundary of a smooth,
convex cone described as in Theorem \ref{main1.1}. Then there exist
some $0< \beta <1$ and some $C > 0$ such that the rescaled function
$\widetilde{u}(x(\xi),s):= u(x(\xi),t(s))\Theta^{-1}(t(s))$
satisfies
\begin{eqnarray}\label{holder}
[D\widetilde{u}]_{\beta} + \left[\frac{\partial
\widetilde{u}}{\partial s}\right]_{\beta} + [\widetilde{k}]_{\beta}
\leq C (||u_{0}|| _{C^{2+\gamma ,1+\frac{\gamma}{2}}(M)},\beta,M),
\end{eqnarray}
where $[f]_{\beta}:=[f]_{x,\beta}+[f]_{s,\frac{\beta}{2}}$ is the
sum of the H\"{o}lder coefficients of $f$ in $M\times [0,S]$ with
respect to $x$ and $s$.
\end{lemma}

\begin{proof}
We divide our proof into  three steps\footnote{~In the proof of
Lemma \ref{lemma4.3}, the constant $C$ may differ from each other.
However, we abuse the symbol $C$ for the purpose of convenience.}.

\textbf{Step 1:} We need to prove that
\begin{equation*}
  [D \widetilde{u}]_{x,\beta}+[D \widetilde{u}]_{s,\frac{\beta}{2}}\leq C(|| u_0||_{ C^{2+\gamma,1+\frac{\gamma}{2}}(M)}, \beta, M).
\end{equation*}
According to Lemmas \ref{estimate1}, \ref{estimate2} and
\ref{estimate3}, it follows that
$$|D \widetilde{u}|+\left|\frac{\partial \widetilde{u}}{\partial s}\right|\leq C(|| u_0||_{ C^{2+\gamma,1+\frac{\gamma}{2}}(M)}, M).$$
Then we can easily obtain the bound of $[\widetilde{u}]_{x,\beta}$
and $[\widetilde{u}]_{s,\frac{\beta}{2}}$ for any $0<\beta<1$. Lemma
\ref{estimate1} in \cite[Chap. 2]{La2} implies that the bound for
$[D\widetilde{u}]_{s,\frac{\beta}{2}}$ follows from a bound for
$[\widetilde{u}]_{s,\frac{\beta}{2}}$ and
$[D\widetilde{u}]_{x,\beta}$. Hence it remains to bound
$[D\widetilde{\varphi}]_{x,\beta}$ since
$D\widetilde{u}=\widetilde{u}D\widetilde{\varphi}$. For this, fix
$s$ and the equation \eqref{Evo-1} can be rewritten as an elliptic
Neumann problem
 \begin{equation}\label{key1}
    -\mbox{div}_{\sigma}\left(\frac{D \widetilde{\varphi}}{\sqrt{1-|D\widetilde{\varphi}|^2}}\right)=\frac{1}{\sqrt{1-|D\widetilde{\varphi}|^2}}+ e^{-\alpha \widetilde{\varphi}} \frac{\sqrt{1-|D\widetilde{\varphi}|^2}}{\widetilde{\varphi}_{s}-1}.
 \end{equation}
In fact, the equation \eqref{key1} is of the form
$D_{\xi}(a^{\xi}(p)+a(\xi,s))=0$. Since $\dot{\widetilde{\varphi}}$
and $|D\widetilde{\varphi}|$ are bounded, we know $a$ is a bounded
function in $\xi$ and $s$. We define $a^{\xi\xi}(p):=\frac{\partial
a^{\xi}}{\partial p^{\xi}}$, the smallest and largest eigenvalues of
$a^{\xi\xi}$ are controlled due to the estimate for
$|D\widetilde{\varphi}|$. By \cite[Chap. 3; Theorem 14.1; Chap. 10,
$\S$2]{La1}, we can get the interior estimate and boundary estimate
of $[D\widetilde{\varphi}]_{x,\beta}$.

\textbf{Step 2:} The next thing to do is to show that
\begin{equation*}
  \left[\frac{\partial \widetilde{u}}{\partial s}\right]_{x,\beta}+\left[\frac{\partial \widetilde{u}}{\partial s}\right]_{s,\frac{\beta}{2}}\leq
  C(||u_0||_{ C^{2+\gamma,1+\frac{\gamma}{2}}(M)}, \beta,
  M).
\end{equation*}
  As $\frac{\partial}{\partial s}\widetilde{u}=\widetilde{u}\left(-\frac{v}{\widetilde{u}^{\alpha+1}\widetilde{k}}+1\right)$, it is enough to bound
  $\left[\frac{v}{\widetilde{u}^{\alpha+1} \widetilde{k}}\right]_{\beta}$.
Set $\widetilde{w}(s):=\frac{v}{\widetilde{u}^{\alpha+1}
\widetilde{k}}= \Theta^{\alpha}\Psi$. Let $\widetilde{\nabla}$ be
the Levi-Civita connection of
$\widetilde{M}_{s}:=\widetilde{X}(M,s)$ w.r.t. the metric
$\widetilde{g}$. Combining with \eqref{div-1} and Lemma
\ref{lemma4.2}, we get
\begin{equation}\label{div-form-02}
\begin{split}
\frac{\partial \widetilde{w}}{\partial s} &=\mbox{div}_{\widetilde{g}} (\widetilde{u}^{-\alpha} \widetilde{k}^{-2} \widetilde{\nabla} \widetilde{w})-2 \widetilde{k}^{-2} \widetilde{u}^{-\alpha} \widetilde{w}^{-1} |\widetilde{\nabla} \widetilde{w}|^2_{\widetilde{g}}\\
&-\alpha\widetilde{w}+\alpha \widetilde{w}^2+ \alpha \widetilde{w}^2
P- \alpha  \widetilde{u}^{-\alpha-1} \widetilde{k}^{-2}
\widetilde{\nabla}_\xi\widetilde{u}
\widetilde{\nabla}^\xi\widetilde{w},
\end{split}
\end{equation}
where $P:=u^{-1}\nabla^\xi u\langle X, X_\xi \rangle_{L}$. Applying
Lemmas \ref{estimate1} and \ref{estimate3}, we have
$$|P|\leq |\nabla u|_g\leq C,$$
where $C$ depends only on $\sup\limits_{M}|Du(\cdot,0)|$, $c_{1}$
and $c_{2}$. The weak formulation of (\ref{div-form-02}) is
\begin{equation}\label{div-form-03}
\begin{split}
\int_{s_0}^{s_1} \int_{\widetilde{M}_s}  \frac{\partial
\widetilde{w} }{\partial s}  \eta d\mu_s ds & =\int_{s_0}^{s_1}
\int_{\widetilde{M}_s} \mbox{div}_{\widetilde{g}}
(\widetilde{u}^{-\alpha} \widetilde{k}^{-2} \widetilde{\nabla}
\widetilde{w}) \eta
-2 \widetilde{k}^{-2} \widetilde{u}^{-\alpha} \widetilde{w}^{-1} |\widetilde{\nabla} \widetilde{w}|^2_{\widetilde{g}} \eta d\mu_s ds\\
&+\int_{s_0}^{s_1} \int_{\widetilde{M}_s}
(-\alpha\widetilde{w}+\alpha \widetilde{w}^2+ \alpha \widetilde{w}^2
P -\alpha  \widetilde{u}^{-\alpha-1} \widetilde{k}^{-2}
\widetilde{\nabla}_\xi\widetilde{u}
\widetilde{\nabla}^\xi\widetilde{w}) \eta d\mu_s ds.
\end{split}
\end{equation}
Since $\nabla_{\mu} \widetilde{\varphi}=0$, the boundary integrals
all vanish, the interior and boundary estimates are basically the
same. We define the test function $\eta:=\zeta^2 \widetilde{w}$,
where $\zeta$ is a smooth function with values in $[0,1]$ and is
supported in a small parabolic neighborhood. Then
\begin{equation}\label{imcf-hec-for-02}
\begin{split}
\int_{s_0}^{s_1} \int_{\widetilde{M}_s}  \frac{\partial
\widetilde{w} }{\partial s}  \zeta^2 \widetilde{w} d\mu_s ds=
\frac{1}{2}||\widetilde{w} \zeta||_{2,\widetilde{M}_s}^2
\Big{|}_{s_0}^{s_1}-\int_{s_0}^{s_1} \int_{\widetilde{M}_s}  \zeta
\dot{\zeta} \widetilde{w}^2 d\mu_s ds,
\end{split}
\end{equation}
where $\dot{\zeta}:=\frac{\partial\xi}{\partial s}$. We have
\begin{equation*}
\begin{split}
&\int_{s_0}^{s_1} \int_{\widetilde{M}_s}  \mbox{div}_{\widetilde{g}} (\widetilde{u}^{-\alpha} \widetilde{k}^{-2} \widetilde{\nabla} \widetilde{w})  \zeta^2 \widetilde{w}d\mu_sds\\
&=\int_{s_0}^{s_1} \int_{\widetilde{M}_s} \mbox{div}_{\widetilde{g}}
(\widetilde{u}^{-\alpha} \widetilde{k}^{-2} \widetilde{\nabla}
\widetilde{w}\zeta^2 \widetilde{w}) d\mu_sds
-\int_{s_0}^{s_1} \int_{\widetilde{M}_s}   \widetilde{u}^{-\alpha} \widetilde{k}^{-2} \zeta^2\widetilde{\nabla}_\xi \widetilde{w}  \widetilde{\nabla}^\xi\widetilde{w}  d\mu_sds\\
&-2\int_{s_0}^{s_1} \int_{\widetilde{M}_s}\widetilde{u}^{-\alpha} \widetilde{k}^{-2} \zeta \widetilde{w}\widetilde{\nabla}_\xi\widetilde{w} \widetilde{\nabla}^\xi \zeta d\mu_sds.\\
\end{split}
\end{equation*}
Using the divergence theorem, we have
\begin{equation*}
\int_{s_0}^{s_1} \int_{\widetilde{M}_s}  \mbox{div}_{\widetilde{g}}
(\widetilde{u}^{-\alpha} \widetilde{k}^{-2} \widetilde{\nabla}
\widetilde{w}\zeta^2 \widetilde{w}) d\mu_sds =-\int_{s_0}^{s_1}
\int_{\partial\widetilde{M}_s} \widetilde{g}(\mu ,
\widetilde{u}^{-\alpha} \widetilde{k}^{-2} \widetilde{\nabla}
\widetilde{w}\zeta^2 \widetilde{w}) d\mu_sds=0.
\end{equation*}
Thus,
\begin{equation*}
\begin{split}
&\int_{s_0}^{s_1} \int_{\widetilde{M}_s}  \mbox{div}_{\widetilde{g}} (\widetilde{u}^{-\alpha} \widetilde{k}^{-2} \widetilde{\nabla} \widetilde{w})  \zeta^2 \widetilde{w}  d\mu_sds\\
&=-\int_{s_0}^{s_1} \int_{\widetilde{M}_s}   \widetilde{u}^{-\alpha}
\widetilde{k}^{-2} \zeta^2\widetilde{\nabla}_\xi \widetilde{w}
\widetilde{\nabla}^\xi\widetilde{w}  d\mu_sds -2\int_{s_0}^{s_1}
\int_{\widetilde{M}_s}\widetilde{u}^{-\alpha} \widetilde{k}^{-2}
\zeta \widetilde{w}\widetilde{\nabla}_\xi\widetilde{w}
\widetilde{\nabla}^\xi \zeta d\mu_sds.
\end{split}
\end{equation*}
Since
\begin{equation*}
\begin{split}
&-\int_{s_0}^{s_1} \int_{\widetilde{M}_s}  \widetilde{u}^{-\alpha} \widetilde{k}^{-2}\widetilde{g}^{\xi\xi}\left(\zeta\widetilde{\nabla}_\xi \widetilde{w}+\widetilde{w}\widetilde{\nabla}_\xi \zeta\right)^{2}  d\mu_sds\\
&=-\int_{s_0}^{s_1} \int_{\widetilde{M}_s}  \left(
\widetilde{u}^{-\alpha} \widetilde{k}^{-2}
\zeta^2\widetilde{\nabla}_\xi \widetilde{w}
\widetilde{\nabla}^\xi\widetilde{w}
 + 2\widetilde{u}^{-\alpha} \widetilde{k}^{-2} \zeta \widetilde{w}\widetilde{\nabla}_\xi\widetilde{w} \widetilde{\nabla}^\xi \zeta + \widetilde{u}^{-\alpha} \widetilde{k}^{-2} |\widetilde{\nabla} \zeta|^2\widetilde{w}^2 \right)d\mu_sds\\
\end{split}
\end{equation*}
is negative, so we can obtain
\begin{equation}\label{imcf-hec-for-03}
\begin{split}
&\int_{s_0}^{s_1} \int_{\widetilde{M}_s}  \mbox{div}_{\widetilde{g}} (\widetilde{u}^{-\alpha} \widetilde{k}^{-2} \widetilde{\nabla} \widetilde{w})  \zeta^2 \widetilde{w}  d\mu_sds\\
&=-\int_{s_0}^{s_1} \int_{\widetilde{M}_s}   \widetilde{u}^{-\alpha}
\widetilde{k}^{-2} \zeta^2\widetilde{\nabla}_\xi \widetilde{w}
\widetilde{\nabla}^\xi\widetilde{w}  d\mu_sds
-2\int_{s_0}^{s_1} \int_{\widetilde{M}_s}\widetilde{u}^{-\alpha} \widetilde{k}^{-2} \zeta \widetilde{w}\widetilde{\nabla}_\xi\widetilde{w} \widetilde{\nabla}^\xi \xi d\mu_sds\\
&\leq\int_{s_0}^{s_1} \int_{\widetilde{M}_s} \widetilde{u}^{-\alpha}
\widetilde{k}^{-2} |\widetilde{\nabla} \zeta|^2\widetilde{w}^2
d\mu_sds.
\end{split}
\end{equation}
We also have
\begin{equation*}
\begin{split}
&\int_{s_0}^{s_1} \int_{\widetilde{M}_s}(-\alpha\widetilde{w}+\alpha \widetilde{w}^2+ \alpha  \widetilde{w}^2 P- \alpha  \widetilde{u}^{-\alpha-1} \widetilde{k}^{-2} \widetilde{\nabla}_\xi\widetilde{u} \widetilde{\nabla}^\xi\widetilde{w})  \zeta^2 \widetilde{w}  d\mu_sds\\
& \leq C|\alpha| \int_{s_0}^{s_1} \int_{\widetilde{M}_s} \zeta^2
(\widetilde{w}^2+|\widetilde{w}|^3)  d\mu_sds+ \int_{s_0}^{s_1}
\int_{\widetilde{M}_s} |\alpha|  \widetilde{u}^{-\alpha-1}
\widetilde{k}^{-2} |\widetilde{\nabla}\widetilde{u}|
|\widetilde{\nabla}\widetilde{w}|  \zeta^2 |\widetilde{w}| d\mu_sds.
\end{split}
\end{equation*}
Using Young's inequality, we can obtain
\begin{equation*}
\begin{split}
&\int_{s_0}^{s_1} \int_{\widetilde{M}_s} |\alpha|  \widetilde{u}^{-\alpha-1} \widetilde{k}^{-2} |\widetilde{\nabla}\widetilde{u}| |\widetilde{\nabla}\widetilde{w}|  \zeta^2 |\widetilde{w}|  d\mu_sds \\
&=\int_{s_0}^{s_1} \int_{\widetilde{M}_s} |\alpha|  (\widetilde{u}^{-\frac{\alpha}{2}} \widetilde{k}^{-1}|\widetilde{\nabla}\widetilde{w}|\zeta)\cdot(\widetilde{u}^{-\frac{\alpha}{2}-1} \widetilde{k}^{-1} |\widetilde{\nabla}\widetilde{u}|\zeta |\widetilde{w}|  )d\mu_sds\\
&\leq  \frac{|\alpha|}{2} \int_{s_0}^{s_1} \int_{\widetilde{M}_s}
\widetilde{u}^{-\alpha} \widetilde{k}^{-2}
|\widetilde{\nabla}\widetilde{w}|^2  \zeta^2  d\mu_sds +
\frac{|\alpha|}{2} \int_{s_0}^{s_1} \int_{\widetilde{M}_s}
\widetilde{u}^{-\alpha-2} \widetilde{k}^{-2}
|\widetilde{\nabla}\widetilde{u}|^2 \zeta^2 \widetilde{w}^2
d\mu_sds,
\end{split}
\end{equation*}
and thus
\begin{equation}\label{imcf-hec-for-04}
\begin{split}
&\int_{s_0}^{s_1} \int_{\widetilde{M}_s}(-\alpha\widetilde{w}+\alpha \widetilde{w}^2+ \alpha  \widetilde{w}^2 P- \alpha  \widetilde{u}^{-\alpha-1} \widetilde{k}^{-2} \widetilde{\nabla}_\xi\widetilde{u} \widetilde{\nabla}^\xi\widetilde{w})  \zeta^2 \widetilde{w}  d\mu_sds\\
& \leq C|\alpha| \int_{s_0}^{s_1} \int_{\widetilde{M}_s} \zeta^2 (\widetilde{w}^2+|\widetilde{w}|^3)  d\mu_sds+ \int_{s_0}^{s_1} \int_{\widetilde{M}_s} |\alpha|  \widetilde{u}^{-\alpha-1} \widetilde{k}^{-2} |\widetilde{\nabla}\widetilde{u}| |\widetilde{\nabla}\widetilde{w}|  \zeta^2 |\widetilde{w}|  d\mu_sds\\
&\leq  C|\alpha| \int_{s_0}^{s_1} \int_{\widetilde{M}_s} \zeta^2
(\widetilde{w}^2+|\widetilde{w}|^3)  d\mu_sds +
\frac{|\alpha|}{2} \int_{s_0}^{s_1} \int_{\widetilde{M}_s}  \widetilde{u}^{-\alpha} \widetilde{k}^{-2}  |\widetilde{\nabla}\widetilde{w}|^2  \zeta^2  d\mu_sds \\
&+ \frac{|\alpha|}{2} \int_{s_0}^{s_1} \int_{\widetilde{M}_s}
\widetilde{u}^{-\alpha-2} \widetilde{k}^{-2}
|\widetilde{\nabla}\widetilde{u}|^2 \zeta^2 \widetilde{w}^2
d\mu_sds.
\end{split}
\end{equation}
Combing (\ref{imcf-hec-for-02}), (\ref{imcf-hec-for-03}) and
(\ref{imcf-hec-for-04}), we have
 \begin{equation*}\label{imcf-hec-for-05}
\begin{split}
&\frac{1}{2}||\widetilde{w} \zeta||_{2,\widetilde{M}_s}^2
\mid_{s_0}^{s_1}
+(2+\frac{\alpha}{2})\int_{s_0}^{s_1} \int_{\widetilde{M}_s} \widetilde{u}^{-\alpha}\widetilde{k}^{-2} |\widetilde{\nabla} \widetilde{w}|^2  \zeta^2   d\mu_sds\\
& \leq \int_{s_0}^{s_1} \int_{\widetilde{M}_s}  \zeta |\dot{\zeta}|
w^2 d\mu_s ds
+\int_{s_0}^{s_1} \int_{\widetilde{M}_s}  \widetilde{u}^{-\alpha} \widetilde{k}^{-2} |\widetilde{\nabla} \zeta|^2\widetilde{w}^2  d\mu_sds\\
&+  C|\alpha| \int_{s_0}^{s_1} \int_{\widetilde{M}_s} \zeta^2 (
\widetilde{w}^2+|\widetilde{w}|^3)  d\mu_sds + \frac{|\alpha|}{2}
\int_{s_0}^{s_1} \int_{\widetilde{M}_s}   \widetilde{u}^{-\alpha-2}
\widetilde{k}^{-2} |\widetilde{\nabla}\widetilde{u}|^2 \zeta^2
\widetilde{w}^2  d\mu_sds,
\end{split}
\end{equation*}
 which implies
 \begin{equation}\label{imcf-hec-for-06}
\begin{split}
&\frac{1}{2}||\widetilde{w} \zeta||_{2,\widetilde{M}_s}^2
\mid_{s_0}^{s_1}
+\frac{(2+\frac{\alpha}{2})}{\max(\widetilde{u}^{\alpha}\widetilde{k}^{2}) }\int_{s_0}^{s_1} \int_{\widetilde{M}_s} |\widetilde{\nabla} \widetilde{w}|^2  \zeta^2   d\mu_sds\\
& \leq (1+ \frac{1}{\min(\widetilde{u}^{\alpha} \widetilde{k}^{2})}) \int_{s_0}^{s_1} \int_{\widetilde{M}_s}  \widetilde{w}^2 (\zeta |\dot{\zeta}| +|\widetilde{\nabla} \zeta|^2)d\mu_s ds\\
&  +  |\alpha| \left(C+ \frac{\max(|
\widetilde{\nabla}\widetilde{u}|)^2}{2\min(\widetilde{u}^{2+\alpha}
\widetilde{k}^{2}) }\right) \int_{s_0}^{s_1} \int_{\widetilde{M}_s}
\left(\zeta^2 \widetilde{w}^2 +\zeta^2 |\widetilde{w}|^3
\right)d\mu_sds.
\end{split}
\end{equation}
This means that $\widetilde{w}$ belong to the De Giorgi class of
functions in $M \times [0,S)$. Similar to the arguments in
\cite[Chap. 5, \S 1 and \S 7]{La2},  there exist  constants
$0<\beta<1$ and $C$ such that
$$[\widetilde{w}]_{\beta}\leq C ||\widetilde{w}||_{L^{\infty}(M \times [0,S))}\leq  C(|| u_0||_{ C^{2+\gamma,1+\frac{\gamma}{2}}(M)}, \beta, M).$$

\textbf{Step 3:} Finally, we have to show that
\begin{equation*}
  [ \widetilde{k}]_{x,\beta}+[\widetilde{k}]_{s,\frac{\beta}{2}}\leq C(|| u_0||_{ C^{2+\gamma,1+\frac{\gamma}{2}}(M)}, \beta, M).
\end{equation*}
This follows from the fact that
$$\widetilde{k}=\frac{\sqrt{1-|D\varphi|^2}}{\widetilde{u}^{\alpha+1} \widetilde{w}}$$
together with the estimates for $\widetilde{u}$, $\widetilde{w}$,
$D\varphi$.
\end{proof}

Then we can obtain the following higher-order estimates:
\begin{lemma}
Let $u$ be a solution to the parabolic system \eqref{Evo-1}, where
$\varphi(x,t)=\ln u(x,t)$, and $\Sigma$ be the boundary of a smooth,
convex cone described as in Theorem \ref{main1.1}. Then for any
$s_0\in (0,S)$, there exist some $0<\beta<1$ and some $C>0$ such
that
\begin{equation}\label{imfcone-holder-01}
||\widetilde{u}||_{C^{2+\beta,1+\frac{\beta}{2}}(M\times [0,S])}\leq
C(|| u_0||_{ C^{2+\gamma, 1+\frac{\gamma}{2}}(M)}, \beta, M)
\end{equation}
and for all $\ell\in \mathbb{N}$,
\begin{equation}\label{imfcone-holder-02}
||\widetilde{u}||_{C^{2\ell+\beta,\ell+\frac{\beta}{2}}(M\times
[s_0,S])}\leq C(||u_0(\cdot,
s_0)||_{C^{2\ell+\beta,\ell+\frac{\beta}{2}}(M)}, \beta, M).
\end{equation}
\end{lemma}

\begin{proof}
By Lemma \ref{lemma2-1}, we have
$$uvk=1+\frac{1}{v^{2}}\varphi_{\xi\xi}=1+u^2 \Delta_g \varphi.$$
Since
$$u^2 \Delta_g \varphi=\widetilde{u}^2 \Delta_{\widetilde{g}} \widetilde{\varphi}=
-| \widetilde{\nabla} \widetilde{u}|^2+ \widetilde{u}
\Delta_{\widetilde{g}} \widetilde{u},$$ then
\begin{equation*}
\begin{split}
\frac{\partial \widetilde{u}}{\partial s}&=\frac{ \partial u}{\partial t} \Theta^{\alpha-1}+\widetilde{u}\\
&=\frac{uvk}{u^{1+\alpha}k^2} \Theta^{\alpha-1} - \frac{2v}{u^{\alpha}k} \Theta^{\alpha-1}+\widetilde{u}\\
&=\frac{\Delta_{\widetilde{g}}
\widetilde{u}}{\widetilde{u}^{\alpha}\widetilde{k}^2}-
\frac{2v}{\widetilde{u}^{\alpha}\widetilde{k}} + \widetilde{u}
+\frac{1-| \widetilde{\nabla}
\widetilde{u}|^2}{\widetilde{u}^{1+\alpha}\widetilde{k}^2},
\end{split}
\end{equation*}
which is a uniformly parabolic equation with H\"{o}lder continuous
coefficients. Therefore, the linear theory (see \cite[Chap.
4]{Lieb}) yields the inequality (\ref{imfcone-holder-01}).

Set $\widetilde{\varphi}=\ln \widetilde{u}$, and then the rescaled
version of the evolution equation in (\ref{Eq-re}) takes the form
\begin{equation*}
  \frac{\partial \widetilde{\varphi}}{\partial s}=- e^{-\alpha\widetilde{\varphi}}\frac{ v^2}{ 1+\frac{1}{v^2} \widetilde{\varphi}_{\xi\xi}}+1,
\end{equation*}
where $v=\sqrt{1-|D \widetilde{\varphi}|^2}$. According to the
$C^{2+\beta,1+\frac{\beta}{2}}$-estimate of $\widetilde{u}$ (see
Lemma \ref{holder}), we can treat the equations for
$\frac{\partial\widetilde{\varphi}}{\partial s}$ and
$D_\xi\widetilde{\varphi}$ as second-order linear uniformly
parabolic PDEs on $M\times [s_0,S]$. At the initial time $s_0$, all
compatibility conditions are satisfied and the initial function
$u(\cdot,t_0)$ is smooth. We can obtain a $C^{3+\beta,
\frac{3+\beta}{2}}$-estimate for $D_\xi \widetilde{\varphi}$ and a
$C^{2+\beta, \frac{2+\beta}{2}}$-estimate for
$\frac{\partial\widetilde{\varphi}}{\partial s}$ (the estimates are
independent of $T$) by Theorem 4.3 and Exercise 4.5 in \cite[Chapter
4]{Lieb}. Higher regularity can be proven by induction over $\ell$.
\end{proof}

\begin{theorem} \label{key-2}
Under the hypothesis of Theorem \ref{main1.1}, we conclude
\begin{equation*}
T^{*}=+\infty.
\end{equation*}
\end{theorem}
\begin{proof}
The proof of this result is quite similar to the corresponding
argument in \cite[Lemma 8]{Mar} and so is omitted.
\end{proof}

\section{Convergence of the rescaled flow} \label{se6}

We know that after the long-time existence of the flow has been
obtained (see Theorem \ref{key-2}), the rescaled version of the
system (\ref{Evo-1}) satisfies
\begin{equation}
\left\{
\begin{aligned}
&\frac{\partial}{\partial
s}\widetilde{\varphi}=\widetilde{Q}(\widetilde{\varphi},D\widetilde{\varphi},
D^2\widetilde{\varphi})  \qquad &&\mathrm{in}~
M\times(0,\infty)\\
&\nabla_{\mu} \widetilde{\varphi}=0  \qquad &&\mathrm{on}~ \partial M\times(0,\infty)\\
&\widetilde{\varphi}(\cdot,0)=\widetilde{\varphi}_{0} \qquad
&&\mathrm{in}~M,
\end{aligned}
\right.
\end{equation}
where
\begin{equation*}
\widetilde{Q}(\widetilde{\varphi},D\widetilde{\varphi},D^2\widetilde{\varphi}):=-e^{-\alpha\widetilde{\varphi}}
\frac{v^2}{1+\frac{1}{v^2}\widetilde{\varphi}_{\xi\xi}}+1
\end{equation*}
and $\widetilde{\varphi}=\ln\widetilde{u}$. Similar to what has been
done in the $C^1$ estimate (see Lemma \ref{estimate3}), we can
deduce a decay estimate of $\widetilde{u}(\cdot, s)$ as follows.

\begin{lemma} \label{lemma5-1}
Let $u$ be a solution of \eqref{Eq2}, then we have
\begin{equation}\label{Gra-est1}
|D\widetilde{u}(x(\xi), t)|\leq\lambda\sup_{M}|D\widetilde{u}(\cdot,
0)|,
\end{equation}
where $\lambda$ is a positive constant depending on $c_{1}$ and
$c_{2}$.
\end{lemma}

\begin{proof}
Set $\widetilde{\psi}=\frac{|D \widetilde{\varphi}|^2}{2}$. Similar
to the argument in Lemma \ref{estimate3}, we can obtain
\begin{equation*}\label{gra--}
\frac{\partial \widetilde{\psi}}{\partial s}=\widetilde{Q}^{\xi\xi}
\widetilde{\psi}_{\xi\xi}+\widetilde{Q}^\xi \widetilde{\psi}_\xi
-\widetilde{Q}^{\xi\xi}\widetilde{\varphi}^{2}_{\xi\xi}+2\alpha(1-\widetilde{Q})\widetilde{\psi},
\end{equation*}
 with the boundary condition
\begin{equation*}
\begin{aligned}
D_ \mu \widetilde{\psi}=0.
\end{aligned}
\end{equation*}
So we have
\begin{equation*}
\left\{
\begin{aligned}
&\frac{\partial \widetilde{\psi}}{\partial s}\leq
\widetilde{Q}^{\xi\xi}\widetilde{\psi}_{\xi\xi}+\widetilde{Q}^\xi\widetilde{\psi}_\xi
\quad &&\mathrm{in}~
M\times(0,\infty)\\
&D_\mu \widetilde{\psi} = 0 \quad &&\mathrm{on}~\partial M\times(0,\infty)\\
&\psi(\cdot,0)=\frac{|D\widetilde{\varphi}(\cdot,0)|^2}{2}
\quad&&\mathrm{in}~M.
\end{aligned}\right.
\end{equation*}
Using the maximum principle and Hopf's lemma, we can get the
gradient estimates of $\widetilde{\varphi}$, and then the inequality
(\ref{Gra-est1}) follows from the relation between
$\widetilde{\varphi}$ and $\widetilde{u}$.
\end{proof}

\begin{lemma}\label{rescaled flow}
Let $u$ be a solution of the flow \eqref{Eq2}. Then,
\begin{equation*}
\widetilde{u}(\cdot, s)
\end{equation*}
converges to a real number as $s\rightarrow +\infty$.
\end{lemma}

\begin{proof}
Set $f(t):= \mathcal{H}^1(M_{t})$, which, as before, represents the
$1$-dimensional Hausdorff measure of $M_{t}$ and is actually the
length of $M_{t}$. The corresponding past-directed timelike unit
normal vector is given by
\begin{eqnarray*}
\nu = \displaystyle{-\frac{1}{v}}\displaystyle{\left(
\frac{u_{\xi}}{u^{2}}\partial_{1} + \partial_{0}\right)},
\end{eqnarray*}
where $v=\sqrt{1 - u^{-2}\left|Du\right|^{2}}$, and the unit normal
vector $\nu$ can be written as
$\nu=\nu^{1}\partial_{1}+\nu^{0}\partial_{0}$ w.r.t. the basis
$\{\partial_{0}=\partial_{r},\partial_{1}=\partial_{\xi}\}$. Then
\begin{eqnarray*}
-\mbox{div}_{M_{t}}\nu = -\left(\frac{\partial\nu^{1}}{\partial
\xi}+\nu^{1}\overline{\Gamma}^{1}_{11}+\nu^{0}\overline{\Gamma}^{1}_{01}\right),
\end{eqnarray*}
 with $\overline{\Gamma}_{IJ}^{K}$ the Christoffel symbols of
$\mathbb{R}_{1}^{n+1}$ w.r.t. the basis
$\{\partial_{0},\partial_{1}\}$. By Lemma \ref{lemma2-1}, we can
obtain
\begin{eqnarray*}
-\mbox{div}_{M_{t}}\nu = \frac{u_{\xi\xi}u
+u^{2}-2u_{\xi}^{2}}{u^{3}v^{3}}=k ,
\end{eqnarray*}
and according to the first variation of a submanifold (see, e.g.,
\cite{ls}), we have
\begin{equation}\label{imcf-crf-for-01}
\begin{split}
f'(t)&=\int_{M_{t}} \mbox{div}_{M_{t}} \left( \frac{\nu}{|X|^{\alpha} k}\right) d\mathcal{H}^1\\
&=\int_{M_{t}} \left\langle \nabla_{e_\xi}\left(\frac{\nu}{|X|^{\alpha} k}\right), e_\xi\right\rangle d\mathcal{H}^1\\
&=-\int_{M_{t}}  |u|^{-\alpha} d\mathcal{H}^{1},
\end{split}
\end{equation}
where $\{e_{\xi}\}$ is an orthonormal basis of the tangent bundle
$TM_{t}$ (i.e., $e_{\xi}=X_{\xi}/|X_{\xi}|$ with $X_{\xi}$ defined
as in Lemma \ref{lemma2-1}). We know that (\ref{C0}) implies
$$\left(-\alpha t+ e^{\alpha \varphi_1}\right)^{-1}\leq u^{-\alpha}\leq \left(-\alpha t+ e^{\alpha \varphi_2}\right)^{-1},$$
where $\varphi_1=\inf_{M^1} \varphi(\cdot,0)$ and
$\varphi_2=\sup_{M^1} \varphi(\cdot,0)$. Hence
$$-\left(-\alpha t+ e^{\alpha \varphi_2}\right)^{-1} f(t)\leq f'(t) \leq -\left(-\alpha t+ e^{\alpha \varphi_1}\right)^{-1}f(t).$$
Combining this fact with (\ref{imcf-crf-for-01}) yields
 \begin{eqnarray*}
\frac{(-\alpha t+ e^{\alpha \varphi_2})^{\frac{1}{\alpha}}
\mathcal{H}^1(M_{0})}{e^{\varphi_2}} \leq f(t)\leq \frac{(-\alpha t+
e^{\alpha \varphi_1})^{\frac{1}{\alpha}} \mathcal{H}^1(M_{0})}{
e^{\varphi_1}}.
 \end{eqnarray*}
 Therefore, the rescaled
hypersurface $\widetilde{M}_s=M_{t} \Theta^{-1}$ satisfies the
following inequality
 \begin{eqnarray*}
\frac{ \mathcal{H}^1(M_{0})}{e^{ \varphi_2}} \leq
\mathcal{H}^1(\widetilde{M}_s)\leq \frac{ \mathcal{H}^1(M_{0})}{ e^{
\varphi_1}},
 \end{eqnarray*}
 which implies that the area
of  $\widetilde{M}_s$ is bounded and the bounds are independent of
$s$. Together with (\ref{imfcone-holder-01}), Lemma \ref{lemma5-1}
and the Arzel\`{a}-Ascoli theorem, we conclude that
$\widetilde{u}(\cdot,s)$ must converge in $C^{\infty}(M)$ to a
constant function $r_{\infty}$ with
\begin{eqnarray*}
\frac{1}{e^{\varphi_{2}}}\left(\frac{\mathcal{H}^1(M_{0})}{\mathcal{H}^1(M)}\right)\leq
r_{\infty}
\leq\frac{1}{e^{\varphi_{1}}}\left(\frac{\mathcal{H}^1(M_{0})}{\mathcal{H}^1(M)}\right),
\end{eqnarray*}
which implies the radius estimate (\ref{radius}).
\end{proof}

So, we have
\begin{theorem}\label{rescaled flow}
The rescaled flow
\begin{equation*}
\frac{d
\widetilde{X}}{ds}=\frac{1}{|\widetilde{X}|^{\alpha}\widetilde{k}}\nu+\widetilde{X}
\end{equation*}
exists for all time and the leaves converge in $C^{\infty}$ to a
piece of hyperbolic plane of center at origin and radius
$r_{\infty}$, i.e., a piece of $\mathscr{H}^{1}(r_{\infty})$, where
$r_{\infty}$ satisfies (\ref{radius}).
\end{theorem}

\vspace{5mm}

\section*{Acknowledgments}
This work is partially supported by the NSF of China (Grant Nos.
11801496 and 11926352), the Fok Ying-Tung Education Foundation
(China) and  Hubei Key Laboratory of Applied Mathematics (Hubei
University).

\end{document}